\documentclass[12pt]{article}

\usepackage{amsmath,amsthm,amsfonts,amssymb,amscd,cite}
\usepackage{graphicx}

\allowdisplaybreaks[4]

\newtheorem {lemma}{Lemma}
\newtheorem {theorem} {Theorem}

\newtheorem {Claim}{Claim}

\allowdisplaybreaks [4]
\usepackage{fullpage}
\usepackage{tikz}

\begin{document}

\title{Girth and Laplacian eigenvalue distribution}

\author{Leyou Xu\footnote{Email: leyouxu@m.scnu.edu.cn}, Bo Zhou\footnote{Email: zhoubo@m.scnu.edu.cn}\\
School of Mathematical Sciences, South China Normal University,\\ Guangzhou 510631, China}

\date{}
\maketitle

\begin{abstract}
Let $G$ be a connected graph of order $n$ with girth $g$. For $k=1,\dots,\min\{g-1, n-g\}$, let $n(G,k)$ be the number of Laplacian eigenvalues (counting multiplicities) of $G$ that fall inside the interval $[n-g-k+4,n]$.
We prove that if $g\ge 4$, then
\[
n(G,k)\le n-g.
\]
Those graphs achieving the bound for $k=1,2$ are determined. We also determine the graphs $G$ with $g=3$ such that $n(G,k)=n-1, n-2, n-3$.
{\bf AMSC:} 05C50, 15A18, 05C38
\end{abstract}

\section{Introduction}

Let $G$ be a graph on $n$ vertices with vertex set $V(G)$ and edge set $E(G)$.
Its Laplacian matrix is the $n\times n$ matrix $L(G)=D(G)-A(G)$, where $A(G)$ is the adjacency matrix of $G$ defined as $(a_{uv})_{u,v\in V(G)}$
with $a_{uv}=1$ if $uv\in E(G)$ and $a_{uv}=0$ otherwise, and $D(G)$ is the diagonal matrix of vertex degrees.
The eigenvalues of $L(G)$ are called the Laplacian eigenvalues of $G$.
The Laplacian eigenvalues of graphs
found applications  in graph theory and in combinatorial optimization \cite{Me,Mo,Zhang}.
Denote by $m_G(\mu)$ the multiplicity of $\mu$ as a Laplacian eigenvalue of a graph $G$. It is known that
$m_G(0)$ is equal to the number of connected components of $G$ \cite{AM}.
For a graph $G$ of order $n$ and an interval $I\subseteq [0,n]$, the number of Laplacian eigenvalues of $G$ in $I$ is denoted by $m_GI$. It is well known that $m_G[0,n]=n$ for a graph $G$ of order $n$ \cite{Me, Moh}.
However, how the Laplacian eigenvalues are distributed in $[0,n]$ is not well understood. For a graph of order $n$, there are connections between some graph parameters and the number of Laplacian eigenvalues in certain subintervals of $[0,n]$. These graph parameters include
the number of pendant vertices \cite{Far}, the number of quasi-pendant vertices \cite{Far,Me91},  the average degree \cite{BRT,JOT,Sin},
 the independence number \cite{AhMS,CMP,COP,GS,Me}, the matching number \cite{GT}, the edge covering number \cite{AAD,GWZF}, the domination number \cite{CJT,GXL,HJT,ZZD}, the chromatic number \cite{AhMS}, and the diameter \cite{AhMS,GM,GXL,XZ1,XZ,XZ3}.

Denote by $K_{p,q}$ the complete bipartite graph with partite sets of size $p$ and $q$. Let $C_n$ denote the cycle on $n$ vertices, $n\ge 3$.

The girth of a graph is the length of a shortest cycle. Recently, Zhen, Wong and Xu \cite{ZWX} established the following result relating girth to Laplacian eigenvalue distribution.

\begin{theorem} \cite{ZWX}
Let $G$ be a connected graph of order $n$ with girth $g=4,\dots, n-1$. Then
\[
m_G(n-g+3,n]\le n-g
\]
with equality if and only if $G\cong K_{2,3}$ or $U_n$, where $U_n$ is obtained  by adding an edge between a vertex on $C_{n-1}$ and a new vertex.
\end{theorem}

Inspired by the above interesting work, we prove the following result. To state the result, we need the following notation.
Let $K_{2,3}^*,K_{2,3}^{**}$ and $K_{2,3}^{***}$ be the graphs given in Fig.~\ref{graphK}.
For $n\ge 6$, let $C_{n-2}=u_1\dots u_{n-2}u_1$ and let $Y_{n,i}$ be the graph obtained from $C_{n-2}$ and two vertices $w, z$ outside $C_{n-2}$ by adding edges $u_1w$ and $u_iz$, where $i=1,\dots,\lceil\frac{n-1}{2} \rceil$.
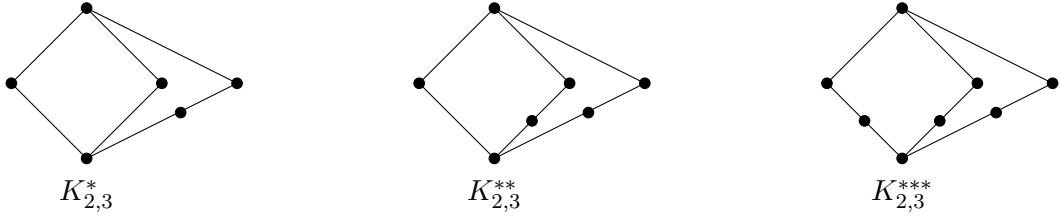
\begin{figure}[hb]
\centering
\begin{minipage}{.32\textwidth}
\centering
\begin{tikzpicture}
\filldraw [black] (1,0) circle (2pt);
\filldraw [black] (-1,0) circle (2pt);
\filldraw [black] (0,1) circle (2pt);
\filldraw [black] (0,-1) circle (2pt);
\filldraw [black] (2,0) circle (2pt);
\draw[black] (0,-1)--(1,0)--(0,1)--(-1,0)--(0,-1);
\draw[black] (0,1)--(2,0)--(0,-1);
\filldraw [black] (1.25,-0.39) circle (2pt);
\node at (0,-1.5) {\small $K_{2,3}^*$};
\end{tikzpicture}
\end{minipage}
\begin{minipage}{.32\textwidth}
\centering
\begin{tikzpicture}
\filldraw [black] (1,0) circle (2pt);
\filldraw [black] (-1,0) circle (2pt);
\filldraw [black] (0,1) circle (2pt);
\filldraw [black] (0,-1) circle (2pt);
\filldraw [black] (2,0) circle (2pt);
\draw[black] (0,-1)--(1,0)--(0,1)--(-1,0)--(0,-1);
\draw[black] (0,1)--(2,0)--(0,-1);
\filldraw [black] (1.25,-0.39) circle (2pt);
\filldraw [black] (0.5,-0.5) circle (2pt);
\node at (0,-1.5) {\small $K_{2,3}^{**}$};
\end{tikzpicture}
\end{minipage}
\begin{minipage}{.32\textwidth}
\centering
\begin{tikzpicture}
\filldraw [black] (1,0) circle (2pt);
\filldraw [black] (-1,0) circle (2pt);
\filldraw [black] (0,1) circle (2pt);
\filldraw [black] (0,-1) circle (2pt);
\filldraw [black] (2,0) circle (2pt);
\draw[black] (0,-1)--(1,0)--(0,1)--(-1,0)--(0,-1);
\draw[black] (0,1)--(2,0)--(0,-1);
\filldraw [black] (1.25,-0.39) circle (2pt);
\filldraw [black] (0.5,-0.5) circle (2pt);
\filldraw [black] (-0.5,-0.5) circle (2pt);
\node at (0,-1.5) {\small $K_{2,3}^{***}$};
\end{tikzpicture}
\end{minipage}
\caption{The graph $K_{2,3}^*$, $K_{2,3}^{**}$ and $K_{2,3}^{***}$.}
\label{graphK}
\end{figure}

\begin{theorem} \label{Gen}
Let $G$ be a connected graph of order $n$ with girth $g$. If $g\ge 4$ and $k=1,\dots,\min\{g-1, n-g\}$,  then
\[
m_G[n-g-k+4,n]\le n-g
\]
with equality when $k=1$ if and only if $g=n-1$ $($i.e., $G\cong K_{2,3}$ or $U_n$$)$, and
with equality when $k=2$ if and only if
\begin{equation} \label{xxxy}
G\cong \begin{cases}
K_{2,4}, K_{2,3}^* & \mbox{if $n=6$},\\
K_{2,3}^{**}, Y_{7,3} & \mbox{if $n=7$},\\
K_{2,3}^{***}, Y_{8,4} & \mbox{if $n=8$},\\
Y_{n,3},\dots, Y_{n,\lceil\frac{n-1}{2} \rceil} & \mbox{if $n\ge 9$}.
\end{cases}
\end{equation}
\end{theorem}

In the previous theorems, only graphs with girth $g\ge 4$ are considered. Next, we state the result for graphs with girth $g=3$.

Let $P_n$ and $K_n$ denote the  path and complete graph on $n$ vertices, respectively.

For vertex disjoint graphs $H$ and $F$, denote by $H\cup F$ the union of them and
 $H\vee F$ the join of them. For a graph $H$ and a positive integer $s$, $sH$ denotes the union of $s$ copies of $H$.

Given a graph $G$, if $E_1\subseteq E(G)$, we denote by $G-E_1$ the graph obtained from $G$ by removing all edges from $E_1$. If $E_1=\{e\}$, we also write $G-e$ for $G-\{e\}$.

For $n\ge 5$,  $H(n)$ denotes the graph obtained from $K_{n-1}$ by removing an edge $e$ and adding an edge between an end of $e$ and a newly added vertex, which is $G_{n,3,2}$ in \cite{XZ}.
For $1\le a\le \lfloor\frac{n-4}{2}\rfloor$,
$H(n,a)$ denotes the graph obtained from vertex disjoint union of a path $P_{4}:=v_1v_2v_3v_4$ and a complete graph $K_{n-4}$ by adding all edges connecting
vertices of $K_{n-4}$ and vertices from $\{v_2, v_3\}$,
$a$ vertices in $K_{n-4}$ and vertex $v_1$, and the remaining vertices in $K_{n-4}$ and  vertex $v_4$, which is $G_{n,3,2,a}$ in \cite{XZ}.

\begin{theorem} \label{Thr}
Let $G$ be a connected graph of order $n\ge 5$ with girth $g=3$.

(i) $m_G(n)=n-1$ if and only if $G\cong K_n$ and
$m_G(n)=n-2$  if and only if $G\cong K_n-e$, where $e$ is an edge of $K_n$.
If $G\not\cong K_n, K_n-e$, then
$m_G(n)\le n-3$ with equality if and only if
$G\cong K_{n-3}\vee 3K_1, K_{n-3}\vee (K_1\cup K_2), K_{n-4}\vee 2K_2$.

(ii) $m_G[n-1,n]=n-1$ if and only if $G\cong K_n$, $m_G[n-1,n]=n-2$ if and only if $G\cong K_n-\{vv_i:i=1,\dots,s\}$ for some $s=1,\dots,n-2$, where $v,v_1,\dots,v_s\in V(K_n)$.
 If $G\ncong K_n,K_n-\{vv_i:i=1,\dots,s\}$ for any $s=1,\dots,n-2$, where $v,v_1,\dots,v_s\in V(K_n)$, then $m_G[n-1,n]\le n-3$ with equality if and only if $G\cong H(n), H(n,a)$ for some $a=1\dots,\lfloor\frac{n-4}{2}\rfloor$, or $G$ is isomorphic to one of the following graphs:

(a) $K_n-\{vv_i:i=1,\dots,s\}-v_1v_2$ for some $s=1,\dots,n-2$, where $v,v_1,\dots,v_s\in V(K_n)$,

(b) $K_n-\{vv_i:i=1,\dots,s\}-\{uu_j:i=1,\dots,t\}$ for some $s=1,\dots,n-3$ and $t=1,\dots,\min\{s,n-2-s\}$, where $v,v_1,\dots,v_s,u,u_1,\dots,u_t\in V(K_n)$,

(c) $K_n-\{vv_i:i=1,\dots,s\}-\{v_1u_j:j=1,\dots,t\}$ for some $s=2,\dots,n-2$ and $t=1,\dots,\min\{s-1, n-1-s\}$, where $v,v_1,\dots,v_s,u_1,\dots,u_t\in V(K_n)$.
\end{theorem}

The rest of this paper is organized as follows. In the next section, we introduce some necessary notation and lemmas. In Sections 3 and 4, we prove Theorem \ref{Gen} for $k=1,2$ and $k\ge 3$, respectively. In Section 5, we prove Theorem \ref{Thr}. Finally, in Section 6, we conclude the paper with some remarks and possible directions for future research.

\section{Preliminaries}

In this section, we review some definitions and lemmas that we need. 

%

Let $G$ be a graph with $S\subset V(G)$ and $u\in V(G)$. We denote by $G[S]$ the subgraph of $G$ induced by $S$ and write $G-u=G[V(G)\setminus \{u\}]$ for the subgraph of $G$ obtained by deleting the vertex $u$ and all its incident edges.
%
If $E_1$ is a set of edges of the complement of $G$, then $G+E_1$ denote the graph obtained from $G$
by adding all the edges from $E_1$.
For $S_1,S_2\subseteq V(G)$, we denote by $E(S_1,S_2)$ the set of edges between vertices in $S_1$ and $S_2$ and write $E(S_1)$ for $E(G[S_1])$.

For $v\in V(G)$, denote by $N_G(v)$ the neighborhood of $v$ in $G$, and $d_G(v)$ denotes the degree of $v$ in $G$.

Fix a graph $G$. For $S\subseteq V(G)$, let $N_S(v)=N_G(v)\cap S$ and $d_S(v)=|N_S(v)|$. If $F$ is a subgraph of $G$, denote $N_{V(F)}(v)$ and $d_{V(F)}(v)$ by $N_F(v)$ and $d_F(v)$.

For a graph $G$ of order $n$, denote by $\mu_1(G), \dots, \mu_n(G)$ the Laplacian eigenvalues of $G$, arranged in nonincreasing order.

The following lemma is well known, see \cite[Corollary 2]{GM} and the subsequent comment on Page 224.

\begin{lemma}\label{d1}\cite{GM}
Let $G$ be a connected graph  on $n$ vertices with maximum degree $\Delta\ge 1$. Then $\mu_1(G)\ge \Delta+1$ with equality if and only if  $\Delta=n-1$.
\end{lemma}

For a bipartite graph $G$ with maximum degree $\Delta$, $\mu_1(G)\le 2\Delta$ with equality when $G$ is connected if and only if $G$ is regular \cite{AM}.
A bipartite graph is semi-regular if vertices from the same partite set have the same degree. We need the following lemma, for which an improvement was given in \cite{Zhang}.

\begin{lemma}\label{d1up}\cite[Theorem 2.9]{Pan}
Let $G$ be a connected graph with at least two edges. Let $r=\max\{d_G(u)+d_G(v): uv\in E(G)\}$. Suppose that $wz\in E(G)$ and $d_G(w)+d_G(z)=r$. Let $s=\max\{d_G(u)+d_G(v): uv\in E(G)\setminus\{wz\}\}$. Then
\[
\mu_1(G)\le 2+\sqrt{(r-2)(s-2)}
\]
with equality if and only if $G$ is a  semi-regular bipartite graph, or $G\cong P_4$.
\end{lemma}


For an $n\times n$ Hermitian matrix $M$, $\rho_i(M)$ denotes its $i$-th largest eigenvalue of $M$ and $\sigma(M)=\{\rho_i(M): i=1,\dots,n \}$ is the spectrum of $M$.
For convenience, if $\rho$ is an eigenvalue of $M$ with multiplicity $s\ge 2$, then we write it as $\rho^{[s]}$ in $\sigma(M)$. For a graph $G$, let $\sigma_L(G)=\sigma(L(G))$. 

\begin{lemma}\cite{BH} \label{sigma}
(i) $\mu_1(P_n)=2+2\cos \frac{\pi}{n}$,\\
(ii) $\sigma_L(C_n)=\{2-2\cos\frac{2j\pi}{n}:j=0,\dots,n-1\}$,\\
(iii) $\sigma(A(P_n))=\{2\cos\frac{j\pi}{n+1}:j=1,\dots,n\}$.
\end{lemma}

\begin{lemma} \cite[Theorem 1.3]{So} \label{cw}
Let $A$ and $B$ be Hermitian matrices of order $n$.
For $1\le i,j\le n$ with $i+j-1\le n$,
\[
\rho_{i+j-1}(A+B)\le \rho_i(A)+\rho_j(B).
\]
\end{lemma}

Let $G$ be a connected  graphs on $n$ vertices with $v\in V(G)$, where $n\ge 3$. Label the vertices of $G$ as $v_1,\dots, v_n$ so that $v_1=v$ and $\{v_2,\dots, v_s\}=N_G(v)$. Let $H=G-v_1$. Then
$L(G)=L(K_1\cup H)+L(K_{1,s}\cup (n-s-1)K_1)$. For $i=1,\dots, n-2$, Lemma \ref{cw} implies that
\[
\mu_{i+1}(G)\le \mu_i(H)+\mu_2(K_{1,s})\le \mu_i(H)+1.
\]

We also need the following two types of interlacing theorem or inclusion principle.

\begin{lemma}\label{interlacing}\cite[Theorem 4.3.28]{HJ}
If $M$ is  a Hermitian matrix of order $n$ and $B$ is its principal submatrix of order $p$, then $\rho_{n-p+i}(M)\le\rho_i(B)\le \rho_{i}(M)$ for $i=1,\dots,p$.
\end{lemma}

\begin{lemma}\label{cauchy} \cite[Theorem 3.2]{Moh}
If $G$ is a graph of order $n$ with $e\in E(G)$, then
\[
\mu_1(G)\ge\mu_1(G-e)\ge \mu_2(G)\ge\dots\ge \mu_{n-1}(G-e)\ge \mu_n(G)=\mu_{n}(G-e)=0.
\]
\end{lemma}

For a graph $G$ of order $n$ and positive integers $k$ and $s$ with $k,s\le n$ and $k-s\ge 1$, this lemma implies that if $E_1\subseteq E(G)$ with $|E_1|=s$, then $\mu_{k}(G)\le \mu_{k-s}(G-E_1)$.

The following lemma is a particular case of \cite[Theorem 3.5]{Das}.

\begin{lemma}\label{secondbound}\cite{Das}
Let $G$ be a connected graph on at least three vertices. Denote by $u$ a vertex of maximum degree and $v$ a vertex of second largest degree $d_2$ of $G$, respectively. If $uv\notin E(G)$ and $N_G(u)\cap N_G(v)=\emptyset$, then $\mu_2(G)\ge d_2+1$.
\end{lemma}

For $v\in V(G)$, let $L_v(G)$ be the principal submatrix of $L(G)$ formed by deleting the row and column corresponding to vertex $v$. 
For a square matrix $M$, the characteristic polynomial of $M$ is defined as $\Phi(M)=\det(xI-M)$.

\begin{lemma}\cite[Lemma 8]{Guo}\label{characteristic}
Let $G$ be a graph and  $v_1v_2$ a cut edge of $G$. Let $G-v_1v_2=G_1\cup G_2$ with $v_1\in V(G_1)$ and $v_2\in V(G_2)$. Then \[
\Phi(L(G))=\Phi(L(G_1))\Phi(L(G_2))-\Phi(L(G_1))\Phi(L_{v_2}(G_2))-\Phi(L_{v_1}(G_1))\Phi(L(G_2)),
\]
where $\Phi(L_{v_i}(G_i))=1$ if $G_i=K_1$ with $i=1,2$.
\end{lemma}

The complement of a graph $G$ is denoted by $\overline{G}$.

\begin{lemma}\label{comp}\cite[Theorem 3.6]{Moh}
Let $G$ be a graph of order $n$. Then
$\mu_i(G)+\mu_{n-i}(\overline{G})=n$ for $i=1,\dots,n-1$.
\end{lemma}

A pendant vertex is a vertex of degree one. A vertex is quasi-pendant if it is adjacent
to a pendant vertex. The following lemma appeared in plain text \cite[p. 263]{Far}.

\begin{lemma} \label{chong} \cite{Far}
Let $G$ be a graph with $p$ pendant vertices and $q$ quasi-pendant vertices. Then
$m_G(1)\ge p-q$.
\end{lemma}

\begin{lemma}\label{min}\cite[Result 4.1]{F}, \cite[Theorem 2.1]{KMNS}
Let $G$ be a connected graph of order $n$ that is not complete. Then $\mu_{n-1}(G)\le\kappa(G)$ with equality  only if $G$ is a join of two graphs, where $\kappa(G)$ is the connectivity of $G$.
\end{lemma}

\section{Proof of Theorem \ref{Gen} for $k=1,2$}

In this section, we establish Theorem~\ref{Gen} in the cases $k=1,2$, which we restate below.

\begin{theorem} \label{XuL}
Let $G$ be a connected graph of order $n$ with girth $g$, where $4\le g\le n-1$. Then
\[
m_G[n-g+3,n]\le n-g
\]
with equality if and only if $g=n-1$ $($i.e., $G\cong K_{2,3}$ or $U_n$$)$. Moreover, if $g\le n-2$, then \[
m_G[n-g+2,n]\le n-g
 \]
with equality if and only if \eqref{xxxy} follows.
\end{theorem}

Before proving the theorem, we present some preparatory lemmas.

For $n\ge 6$, let $U_n'$ be the graph obtained from $C_{n-2}\cup K_2$ by adding an edge between one vertex of $C_{n-2}$ and one vertex of $K_2$.

\begin{lemma}\label{u1}
If $G\cong U_n$ with $n\ge 5$ or $G\cong U_n'$ with $n\ge 6$, then 
$\mu_2(G)<4$.
\end{lemma}
\begin{proof}
Denote by $u$ the vertex of degree three in $G$ and $v$ one of its neighbor (on the cycle) of degree two. As $G-uv\cong P_n$, we have by Lemmas \ref{cauchy} and \ref{sigma} (i) that $\mu_2(G)\le \mu_1(P_n)=2+2\cos \frac{\pi}{n}<4$.
\end{proof}

\begin{lemma}\label{cal2} For $n\ge 6$,
 $\mu_2(Y_{n,1})<4$.
\end{lemma}
\begin{proof}
It is easy to check that $\mu_2(Y_{n,1})<4$ for $n=6,\dots, 10$. Suppose that $n\ge 11$.
Note that $u_1w$ is a cut edge of $Y_{n,1}$ and $Y_{n,1}-u_1w\cong U_{n-1}\cup K_1$.
By Lemma \ref{characteristic}, we have \[
\Phi(L(Y_{n,1}))=\Phi(L(U_{n-1}))x-\Phi(L(U_{n-1}))-\Phi(L_{u_1}(U_{n-1}))x.
\]
As $u_1z$ is also a cut edge of $U_{n-1}$, we have by Lemma \ref{characteristic} again that \[
\Phi(L(U_{n-1}))=\Phi(L(C_{n-2}))x-\Phi(L(C_{n-2}))-\Phi(L_{u_1}(C_{n-2}))x.
\]
Note that $\Phi(L_{u_1}(U_{n-1}))=(x-1)\Phi(L_{u_1}(C_{n-2}))$. Then
\begin{align*}
\Phi(L(Y_{n,1}))&=(x-1)\Phi(L(U_{n-1}))-\Phi(L_{u_1}(U_{n-1}))x\\
&=(x-1)^2\Phi(L(C_{n-2}))-2(x-1)x\Phi(L_{u_1}(C_{n-2})).
\end{align*}
By Lemma \ref{sigma} (ii), $\Phi(L(C_{n-2}))=\prod_{i=0}^{n-3}(x-2+2\cos\frac{2i\pi}{n-2})$. From
$L_{u_1}(C_{n-2})=2I-A(P_{n-3})$ and Lemma \ref{sigma} (iii),  we have $
\Phi(L_{u_1}(C_{n-2}))=\prod_{i=1}^{n-3}\left(x-2+2\cos \frac{i\pi}{n-2}\right)$.
Then
\begin{align*}
\Phi(L(Y_{n,1}))&=(x-1)^2\Phi(L(C_{n-2}))-2(x-1)x\Phi(L_{u_1}(C_{n-2}))\\
&=(x-1)^2\prod_{i=0}^{n-3}\left(x-2+2\cos\frac{2i\pi}{n-2}\right)-2(x-1)x\prod_{i=1}^{n-3}\left(x-2+2\cos \frac{i\pi}{n-2}\right)\\
&=(x-1)xf(x),
\end{align*}
where
\[
f(x)=(x-1)\prod_{i=1}^{n-3}\left(x-2+2\cos \frac{2i\pi}{n-2}\right)-2\prod_{i=1}^{n-3}\left(x-2+2\cos\frac{i\pi}{n-2} \right).
\]

Suppose that $n$ is even. Then
\[
f(x)=\prod_{i=1}^{n/2-2}\left(x-2+2\cos \frac{2i\pi}{n-2}\right)g_1(x),
\]
where \[
g_1(x)=(x-1)(x-4)\prod_{i=n/2}^{n-3}\left(x-2+2\cos \frac{2i\pi}{n-2}\right)-2\prod_{i=0}^{n/2-2}\left(x-2+2\cos \frac{(2i+1)\pi}{n-2}\right).
\]
Note that \[
g_1\left(2+2\cos  \frac{2\pi}{n-2}\right)=-2\prod_{i=0}^{n/2-2}\left(2\cos  \frac{2\pi}{n-2}+2\cos \frac{(2i+1)\pi}{n-2}\right)>0.
\]
Then $2+2\cos \frac{2\pi}{n-2}$ is a root of $f(x)=0$, but not a root of $g_1(x)=0$.
As \[
g_1(4)=-2\prod_{i=0}^{n/2-2}\left(2+2\cos \frac{(2i+1)\pi}{n-2}\right)<0,
\]
there is a root of $g_1(x)=0$ in the interval $(2+2\cos\frac{2\pi}{n-2},4)$.
As $Y_{n,1}-u_1u_2-u_1u_{n-2}\cong P_{n-3}\cup P_3$, we have by Lemma \ref{cauchy} that
\[
\mu_4(Y_{n,1})\le \mu_2(P_{n-3}\cup P_3)=2+2\cos \frac{2\pi}{n-3}<2+2\cos \frac{2\pi}{n-2}.
\]
By Lemma \ref{d1}, $\mu_1(Y_{n,1})\ge 5$.
So $\mu_2(Y_{n,1})<4$.

Suppose next that $n$ is odd. Then \[
f(x)=\prod_{i=1}^{(n-3)/2}\left(x-2+2\cos \frac{2i\pi}{n-2}\right)g_2(x),
\]
where \[
g_2(x)=(x-1)\prod_{i=(n-1)/2}^{n-3}\left(x-2+2\cos \frac{2i\pi}{n-2}\right)-2\prod_{i=0}^{(n-5)/2}\left(x-2+2\cos \frac{(2i+1)\pi}{n-2}\right).
\]
Then $2+2\cos\frac{\pi}{n-2}$ is a root of $f(x)=0$.
As $Y_{n,1}-u_1u_2-u_1u_{n-2}\cong P_{n-3}\cup P_3$, we have by Lemmas \ref{cauchy} and \ref{sigma} (i) that \[
\mu_3(Y_{n,1})\le \mu_1(P_{n-3}\cup P_3)=2+2\cos \frac{\pi}{n-3}<2+2\cos\frac{\pi}{n-2}.
\]
Recall that  $\mu_1(Y_{n,1})\ge 5$. So $\mu_2(Y_{n,1})=2+2\cos \frac{\pi}{n-2}<4$, as desired.
\end{proof}

\begin{lemma} \label{za} For $n\ge 6$, $\mu_2(Y_{n,2})<4$, and for $n=6,8$, $\mu_2(Y_{n,3})<4$.
\end{lemma}

\begin{proof}
As $Y_{n,2}-u_1u_2\cong P_n$, we have by Lemmas \ref{cauchy} and \ref{sigma} (i) that $\mu_2(Y_{n,2})\le \mu_1(P_n)<4$. If $n=6,8$, then it can be checked by a direct calculation that $\mu_2(Y_{n,3})<4$.
\end{proof}

Now, we are ready to prove Theorem \ref{XuL}.

\begin{proof}[Proof of Theorem \ref{XuL}]

\begin{Claim} \label{c1}
If $g=n-1$, then $m_G[n-g+3,n]= n-g$.
\end{Claim}

\begin{proof}
Suppose that $g=n-1$. Then $G\cong K_{2,3}$ (with $n=5$)  or $U_n$.
In the former case, as $\sigma_L(K_{2,3})=\{5,3,2^{[2]},0\}$, we have
$m_G[n-g+3,n]=m_{K_{2,3}}[4,5]=1=n-g$. In the latter case, we have  $\mu_1(U_n)>4$ by Lemma \ref{d1} and
 $\mu_2(U_n)<4=n-g+3$ by Lemma \ref{u1}, so
 $m_G[n-g+3,n]=m_{U_n}[4,n]=1=n-g$.
\end{proof}

By Claim \ref{c1} and noting that $m_G[n-g+3,n]\le m_G[n-g+2,n]$ for $g\le n-2$, we only need to show the second part of Theorem \ref{XuL}. Suppose that $g\le n-2$.

\begin{Claim} \label{c2}
If $G$ is a graph given in \eqref{xxxy}, then $m_G[n-g+2,n]\ge n-g$.
\end{Claim}

\begin{proof} Suppose that $G$ is a graph given in \eqref{xxxy}.  Firstly, we show that
$\mu_2(G)\ge 4$.
By a direct calculation, $\mu_2(K_{2,4})=\mu_2(K_{2,3}^*)=\mu_2(K_{2,3}^{***})=4$, $\mu_2(K_{2,3}^{**})=4.414$,
and $\mu_2(Y_{n,3})\ge 4$ if $n=7,9,10$. If $n\ge 11$, then $Y_{n,3}$ contains $F$ as a subgraph (see Fig. \ref{graphF}), so we have by Lemma \ref{cauchy} that $\mu_2(Y_{n,3})\ge \mu_2(F)=4.01>4$.
Suppose that $n\ge 8$.  Note that $u_1$ and $u_i$ the vertices of maximum and the second largest degree in $Y_{n,i}$ for $i=4,\dots, \lceil\frac{n-1}{2}\rceil$. As $u_1u_i\notin E(Y_{n,i})$ and $N_{Y_{n,i}}(u_1)\cap N_{Y_{n,i}}(u_i)=\emptyset$, we have by Lemma \ref{secondbound} that $\mu_2(Y_{n,i})\ge 4$. It follows that $\mu_2(G)\ge 4$.

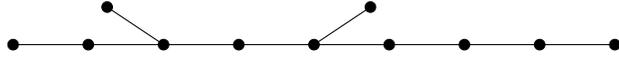
\begin{figure}[htbp]
\centering
\begin{tikzpicture}
\filldraw [black] (5,0) circle (2pt);
\filldraw [black] (4,0) circle (2pt);
\filldraw [black] (3,0) circle (2pt);
\filldraw [black] (2,0) circle (2pt);
\filldraw [black] (1,0) circle (2pt);
\filldraw [black] (0,0) circle (2pt);
\filldraw [black] (-1,0) circle (2pt);
\filldraw [black] (-2,0) circle (2pt);
\filldraw [black] (-3,0) circle (2pt);
\filldraw [black] (1.75,0.5) circle (2pt);
\filldraw [black] (-1.75,0.5) circle (2pt);
\draw[black] (5,0)--(-3,0);
\draw[black] (1.75,0.5)--(1,0);
\draw[black] (-1.75,0.5)--(-1,0);
\end{tikzpicture}
\caption{The graph $F$.}
\label{graphF}
\end{figure}

Secondly, as $g=n-2$ and $\mu_2(G)\ge 4$, we have
$m_G[n-g+2,n]=m_G[4,n]\ge 2=n-g$.
\end{proof}

By the fact  that $m_G[n-g+2,n]\le n-g$ if and only if $\mu_{n-g+1}(G)<n-g+2$, together with
Claim \ref{c2} (as well as the trivial fact that $\mu_{n-g+1}(G)\le \mu_{n-g}(G)$),  it suffices to show the following two items:

\begin{enumerate}

\item[(i)]  If $G$ is  a graph in \eqref{xxxy}, then $\mu_{n-g+1}(G)<n-g+2$.

\item[(ii)]  If $G$ is not (isomorphic to) any graph in \eqref{xxxy}, then $\mu_{n-g}(G)<n-g+2$.
\end{enumerate}

Let $C:=v_1v_2\dots v_gv_1$ be a shortest cycle of $G$ and $S=V(G)\setminus V(C)$.

\noindent
{\bf Case 1.} There is some vertex in $S$, say $w$, which has no neighbors on $C$.

Note that $G$ is not any graph given in \eqref{xxxy}. It suffices to prove that $\mu_{n-g}(G)<n-g+2$.

As $G$ is connected, there is a vertex $v\in S$ that is adjacent to some vertex on $C$. Assume that $vv_1\in E(G)$. Let $H=G[V(C)\cup \{v,w\}]$ and let $B$ be the principal submatrix of $L(G)$ corresponding to vertices of $H$. Let $S'=S\setminus \{v,w\}$.
Then $B=L(H)+M$ with \[
M=\mbox{diag}(d_{S'}(v_1),\dots,d_{S'}(v_g),d_{S'}(v),d_{S'}(w)),
\]
Evidently, $\rho_1(M)=\max\{d_{S'}(z): z\in V(C)\cup \{v,w\}\}\le |S'|=n-g-2$.

\noindent
{\bf Case 1.1.} $d_C(v)=1$.

By Lemmas \ref{interlacing} and \ref{cw}, we have
\[
\mu_{n-g}(G)=\rho_{n-(g+2)+2}(L(G))\le \rho_{2}(B)\le \mu_2(H)+\rho_1(M)\le \mu_2(H)+n-g-2.
\]
Corresponding to the cases  $vw\notin E(G)$ or $vw\in E(G)$, we have $H\cong U_{g+1}\cup K_1$ or
$H\cong U_{g+2}'$, so we have $\mu_2(H)=\mu_2(U_{g+1}\cup K_1)=\mu_2(U_{g+1})$
or $\mu_2(H)=\mu_2(U_{g+2}')$.
By Lemma \ref{u1}, $\mu_2(H)<4$. So
\[
\mu_{n-g}(G)<4+n-g-2=n-g+2.
\]

\noindent
{\bf Case 1.2.}  $d_C(v)\ge 2$.

Assume that  $vv_i\in E(G)$ for some $i=2,\dots,g$. As $g>3$, we have $i=3,\dots,g-1$.
As $v_1\dots v_ivv_1$ is a cycle of length $i+1$ and
$v_1vv_i\dots v_gv_1$ is a cycle of length $g-i+3$, we have $i+1\ge g$ and $g-i+3\ge g$, so  $i=3$ and $g=4$. Then $H\cong K_{2,3}\cup K_1$ if $vw\notin E(G)$, and $H-vw\cong K_{2,3}\cup K_1$ if $vw\in E(G)$. By a direct calculation, $\mu_2(H)= 3$.
By Lemmas \ref{interlacing} and  \ref{cw},
\[
\mu_{n-g}(G)\le \rho_2(B)\le \mu_2(H)+\rho_1(M)\le 3+n-6=n-3<n-2=n-g+2,
\]

Combining Cases 1.1 and 1.2, we have $\mu_{n-g}(G)<n-g+2$, proving Item (ii).

\noindent
{\bf Case 2.} Every vertex in $S$ is adjacent to at least one vertex on $C$.

\noindent
{\bf Case 2.1.} There is a vertex in $S$, say $u$, adjacent to at least two vertices on $C$.

In this case, $G$ is not  any graph given in \eqref{xxxy} except $K_{2,4}$. If $G\cong K_{2,4}$, then
$n=6$, $g=4$, and we have
\[
\mu_{n-g+1}(G)=\mu_3(G)=2<4=n-g+2.
\]
Suppose that $G\not\cong K_{2,4}$.

Assume that $uv_1, uv_i\in E(G)$ for some $i=2,\dots,g$. As $g>3$, $i=3,\dots,g-1$. From $i+1\ge g$ and $g-i+3\ge g$, we have  $i=3$ and $g=4$. As $g\le n-2$, $S\setminus \{u\}\ne \emptyset$.

Suppose first that there is a vertex $z\in S\setminus \{u\}$ adjacent to  $v_2$ or $v_4$.  Let $H^*=G[V(C)\cup \{u,z\}]$ and $B^*$ be the principal submatrix of $L(G)$ corresponding to vertices of $H^*$. Note that  $H^*$ is a spanning subgraph of $K_{3,3}$ and $\sigma_L(K_{3,3})=\{6,3^{[4]},0\}$.  By Lemma \ref{cauchy},  $\mu_2(H^*)\le \mu_2(K_{3,3})= 3$. By Lemmas \ref{interlacing} and \ref{cw},
\[
\mu_{n-g}(G)=\rho_{n-6+2}(L(G))\le \rho_2(B^*)\le \mu_2(H^*)+n-6\le 3+n-6=n-3<n-g+2.
\]

Suppose next that each vertex in $S\setminus\{u\}$ is adjacent to exactly one of $v_1$ or $v_3$.

Suppose that $S$ is not independent.  Then there is a vertex $z_1\in N_S(v_1)$ adjacent to some $z_2\in N_S(v_3)$. It then follows that $z_1v_3,z_2v_1\notin E(G)$. Let $H=G[V(C)\cup \{u,z_1,z_2\}]$. Evidently, $H$ is the graph obtained from $K_{2,4}$ by subdivision of an edge. By a direct calculation, $\mu_3(H)=2.382<3$. Let $S''=S\setminus\{u,z_1, z_2\}$.
Denote by $B$  the principal submatrix of $L(G)$ corresponding to vertices of $H$. Then
\[
B=L(H)+\mbox{diag}(d_{S''}(v_1),\dots, d_{S''}(v_4), d_{S''}(u), d_{S''}(z_1), d_{S''}(z_2)).
\]
By Lemmas \ref{interlacing} and \ref{cw} that
\[
\mu_{n-g}(G)=\mu_{n-4}(G)\le\rho_3(B)\le\mu_3(H)+|S''|<3+n-7=n-4<n-g+2.
\]
Suppose that $S$ is independent. Then $G$ is a spanning subgraph of $K_{2,n-2}$. As $\sigma_L(K_{2,n-2})=\{n,n-2,2^{[n-3]},0\}$, we have
\[
\mu_{n-g}(G)\le \mu_{n-4}(K_{2,n-2})=2<n-g+2 \mbox{ if $n\ge 7$}.
\]
If $n=6$, then $G\cong  K_{2,4}^-$ (the graph obtained from $K_{2,4}$ by deleting an edge), so
\[
\mu_{n-g}(G)=\mu_2(G)=3.572<4=n-g+2.
\]

Now Items (i)and (ii) follow.

\noindent
{\bf Case 2.2.} Every vertex in $S$ adjacent to exactly one vertex on $C$.

In this case, $G\ncong K_{2,4}$.
Suppose first that
$G$ is one graph given in \eqref{xxxy}. 
Then $g=n-2$. Let $u$ and $z$ be the two vertices in $S$.
Denote by $v_i$ the neighbor of $z$ in $G$.
Note that
\[
G-zv_i\cong \begin{cases}
U_{g+1}\cup K_1&\mbox{ if }uz\notin E(G),\\
U_{g+2}'&\mbox{ if }uz\in E(G).
\end{cases}
\]
By  Lemmas \ref{cauchy} and \ref{u1},
 \[
\mu_{n-g+1}(G)=\mu_3(G)\le \mu_2(G-zv_i)\le \max\{\mu_2(U_{g+1}), \mu_2(U_{g+2}')\}<4=n-g+2,
\]
so Item (i) follows.

We are left to show Item (ii), which is the following.

\begin{Claim}\label{c3}
If $G$ is not given in \eqref{xxxy}, then $\mu_{n-g}(G)<n-g+2$.
\end{Claim}
\begin{proof}
Suppose that $S$ is independent. Suppose that $g=n-2$. Then
$G\cong Y_{n,i}=C_{n-2}+u_1w+u_iz$ for some $i=1,\dots,\lceil\frac{n-1}{2} \rceil$,  $\{w, z\}=S$.
As $G$ is not isomorphic to any graph given in \eqref{xxxy},
$G\cong Y_{n,1}, Y_{n,2}, Y_{6,3}, Y_{8,3}$. By Lemmas \ref{cal2} and \ref{za},
\[
\mu_{n-g}(G)=\mu_2(G)<4=n-g+2.
\]
Suppose  that $g\le n-3$.
If there is no vertex on $C$ with at least two neighbors in $S$, then the set of edges outside $C$ is  a matching of size $n-g$. Let $E_1$ be a subset of this matching with $|E_1|=n-g-1$. Then $G-E_1\cong U_{g+1}\cup (n-g-1)K_1$. So by Lemmas \ref{cauchy} and \ref{d1up},
\[
\mu_{n-g}(G)\le \mu_1(G-E_1)=\mu_1(U_{g+1})<5\le n-g+2.
\]
Suppose that there is a vertex on $C$, say $v_1$, having at least two neighbors $w_1,w_2\in S$. Let $w_3\in S\setminus \{w_1,w_2\}$ and let $v_j$ be the neighbor of $w_3$ on $C$. Let $H_1=G[(V(C)\setminus \{v_1\})\cup \{w_1,w_2,w_3\}]$. Note that $H_1\cong P_{g-1}\cup 3K_1$ if $j=1$ and $H_1-v_jw_3\cong P_{g-1}\cup 3K_1$ if $j\ne 1$. By Lemmas \ref{cauchy} and \ref{sigma} (i), we have
 $\mu_2(H_1)\le \mu_1(H_1)=\mu_1(P_{g-1})<4$ or $\mu_2(H_1)\le \mu_1(H_1-v_jw_3)=\mu_1(P_{g-1})<4$.
Denoting by $B_1$ the principal submatrix of $L(G)$ corresponding to vertices of $H_1$, we have by Lemmas \ref{interlacing} and \ref{cw} that \[
\mu_{n-g}(G)\le \rho_2(B_1)\le  \mu_2(H_1)+n-g-2<n-g+2.
\]

Suppose that $S$ is not independent, say $u_1u_2\in E(G)$ with $u_1,u_2\in S$. Assume that $u_1v_1,u_2v_j\in E(G)$ with $j-1\le g+1-j$. As $j+2\ge g$, we have $(j,g)=(2,4),(3,4),(3,5)$ or $(4,6)$. If $(j,g)=(2,4)$, then $G[V(C)\cup \{u_1,u_2\}]\cong R_1$ (see Fig. \ref{r123}), so by Lemmas \ref{interlacing} and \ref{cw} and the fact that $\mu_2(R_1)=3$, we have \[
\mu_{n-g}(G)\le \mu_2(R_1)+n-6= n-3<n-g+2.
\]
Suppose in the following that $(j,g)\ne (2,4)$.
As $G\ncong K_{2,3}^*,K_{2,3}^{**},K_{2,3}^{***}$, $n\ge g+3$. Let $u_3\in S\setminus \{u_1,u_2\}$ and  $H_2=G[V(C)\cup \{u_1,u_2,u_3\}]$. Let $B_2$ be the principal submatrix of $L(G)$ corresponding to vertices of $H_2$. By Lemmas \ref{interlacing} and \ref{cw},
\[
\mu_{n-g}(G)\le \rho_{3}(B_2)\le \mu_3(H_2)+n-g-3.
\]
In the following we show that $\mu_3(H_2)<5$.

\begin{figure}[htbp]
\centering
\begin{tikzpicture}
\centering
\filldraw [black] (3.5,1) circle (2pt);
\filldraw [black] (3.5,0) circle (2pt);
\filldraw [black] (2.5,1) circle (2pt);
\filldraw [black] (2.5,0) circle (2pt);
\filldraw [black] (4.5,1) circle (2pt);
\filldraw [black] (4.5,0) circle (2pt);
\draw[black] (2.5,1)--(4.5,1)--(4.5,0)--(2.5,0)--(2.5,1);
\draw[black] (3.5,1)--(3.5,0);
\node at (3.5,-0.5) {\small $R_1$};
\filldraw [black] (7,1) circle (2pt);
\filldraw [black] (7,0) circle (2pt);
\filldraw [black] (8,1) circle (2pt);
\filldraw [black] (8,0) circle (2pt);
\filldraw [black] (6,1) circle (2pt);
\filldraw [black] (6,0) circle (2pt);
\filldraw [black] (7,1.5) circle (2pt);
\draw[black] (6,1)--(8,1)--(8,0)--(6,0)--(6,1);
\draw[black] (7,1)--(7,0);
\draw[black] (6,1)--(7,1.5)--(8,0);
\node at (7,-0.5) {\small $R_2$};
\filldraw [black] (9.5,0.5) circle (2pt);
\filldraw [black] (10.5,0.5) circle (2pt);
\filldraw [black] (10,0) circle (2pt);
\filldraw [black] (10,1) circle (2pt);
\filldraw [black] (11.5,0) circle (2pt);
\filldraw [black] (11.5,1) circle (2pt);
\filldraw [black] (10.8,1.5) circle (2pt);
\draw[black] (9.5,0.5)--(10,0)--(10.5,0.5)--(10,1)--(9.5,0.5);
\draw[black] (10,0)--(11.5,0)--(11.5,1)--(10,1)--(10.8,1.5)--(11.5,0);
\node at (10.5,-0.5) {\small $R_3$};
\end{tikzpicture}
\caption{The graphs $R_1$, $R_2$ and $R_3$.}
\label{r123}
\end{figure}
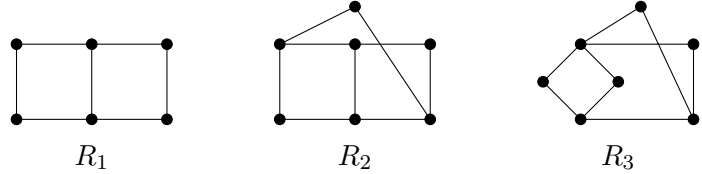

Suppose that $(j,g)=(3,4)$.
Recall that $u_3$ is adjacent to exactly one vertex on $C$.
Then $H_2$ is isomorphic to a spanning subgraph of $R_2$ when $u_3v_2\in E(G)$ or $u_3v_4\in E(G)$, and a spanning subgraph of $R_3$ (see Fig. \ref{r123}) when $u_3v_1\in E(G)$ or $u_3v_3\in E(G)$. By Lemma \ref{cauchy} and a direct calculation,
$\mu_3(H_2)\le \max\{\mu_3(R_i): i=2,3\}<5$.

Suppose that $(j,g)=(3,5)$. If $u_3$ is adjacent to $u_1$ or $u_2$, say $u_1\in E(G)$, then $u_2u_3\notin E(G)$ and $N_C(u_3)=\{v_4\}$ as $G$ has no cycle of length $3$ or $4$, which implies that $H_2\cong F_0$ (see Fig. \ref{f12}) and $\mu_3(H_2)=4<5$ by a direct calculation. If $u_3$ is not adjacent to either of $u_1$ and $u_2$, then $H_2-v_2\cong U_7$ if $u_3v_2\notin E(G)$ and $H_2-v_2\cong C_6\cup K_1$ if $u_3v_2\in E(G)$, which follows by Lemmas \ref{cw}, \ref{sigma} (ii) and \ref{u1} that $\mu_3(H_2)\le \mu_2(H_2-v_2)+1<5$.

Suppose that $(j,g)=(4,6)$. As $G$ has no cycle of length $4,5$, $u_3$ can not be  adjacent to $u_1$ or $u_2$. Then $H_3$ is isomorphic to $F_1$ or $F_2$ in Fig. \ref{f12} and so $\mu_3(H_2)<5$ by a direct calculation.

\begin{figure}[htbp]
\centering
\begin{tikzpicture}
\filldraw [black] (-5,1) circle (2pt);
\filldraw [black] (-4.2,0.1) circle (2pt);
\filldraw [black] (-4.5,-1) circle (2pt);
\filldraw [black] (-5.8,0.1) circle (2pt);
\filldraw [black] (-5.5,-1) circle (2pt);
\filldraw [black] (-6.5,0.2) circle (2pt);
\filldraw [black] (-3.5,-0.3) circle (2pt);
\filldraw [black] (-3.5,0.7) circle (2pt);
\draw[black] (-5,1)--(-4.2,0.1)--(-4.5,-1)--(-5.5,-1)--(-5.8,0.1)--(-5,1)--(-3.5,0.7)--(-3.5,-0.3)--(-4.5,-1);
\draw[black](-3.5,0.7)--(-6.5,0.2)--(-5.5,-1);
\node at(-5,-1.4) {\small $F_0$};
\centering
\filldraw [black] (0,1) circle (2pt);
\filldraw [black] (0,-1) circle (2pt);
\filldraw [black] (0.6,0.3) circle (2pt);
\filldraw [black] (0.6,-0.3) circle (2pt);
\filldraw [black] (-0.6,0.3) circle (2pt);
\filldraw [black] (-0.6,-0.3) circle (2pt);
\filldraw [black] (1.2,0.3) circle (2pt);
\filldraw [black] (1.2,-0.3) circle (2pt);
\filldraw [black] (0.8,1) circle (2pt);
\draw[black] (0,1)--(0.6,0.3)--(0.6,-0.3)--(0,-1)--(-0.6,-0.3)--(-0.6,0.3)--(0,1)--(0.8,1);
\draw[black] (0,1)--(1.2,0.3)--(1.2,-0.3)--(0,-1);
\node at(0,-1.4) {\small $F_1$};
\filldraw [black] (5,1) circle (2pt);
\filldraw [black] (5,-1) circle (2pt);
\filldraw [black] (5.6,0.3) circle (2pt);
\filldraw [black] (5.6,-0.3) circle (2pt);
\filldraw [black] (4.4,0.3) circle (2pt);
\filldraw [black] (4.4,-0.3) circle (2pt);
\filldraw [black] (6.2,0.3) circle (2pt);
\filldraw [black] (6.2,-0.3) circle (2pt);
\filldraw [black] (3.6,0.3) circle (2pt);
\draw[black] (5,1)--(5.6,0.3)--(5.6,-0.3)--(5,-1)--(4.4,-0.3)--(4.4,0.3)--(5,1);
\draw[black] (5,1)--(6.2,0.3)--(6.2,-0.3)--(5,-1);
\draw[black] (4.4,0.3)--(3.6,0.3);
\node at(5,-1.4) {\small $F_2$};
\end{tikzpicture}
\caption{The graphs $F_0$, $F_1$ and $F_2$.}
\label{f12}
\end{figure}
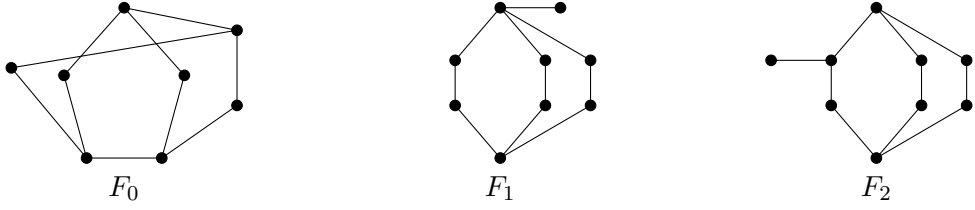

Now we have proved that $\mu_3(H_2)<5$, so $\mu_{n-g}(G)\le \mu_3(H_2)+n-g-3<n-g+2$.
\end{proof}

This completes the proof.
\end{proof}

By Theorem \ref{XuL}, if $G$ a connected graph of order $n$ with girth $g$, where $4\le g\le n-2$, then
\[
m_G[n-g+3,n]\le n-g-1.
\]
By Lemmas \ref{d1} and \ref{cal2}, $m_{Y_{n,1}}[n-g+3,n]=n-g-1$. Moreover, for $g\ge 5$, it may be proved that
$Y_{n,1}$ is the unique graph achieving this bound.

\section{Proof of Theorem \ref{Gen} for $k\ge 3$}

In this section, we establish Theorem~\ref{Gen} for  $k\ge 3$, which is  stated as follows.

\begin{theorem} \label{Gen22}
Let $G$ be a connected graph of order $n$ with girth $g$. For $k=3,\dots,g-1$, if $4\le g\le n-k$, then
\[
m_G[n-g-k+4,n]\le n-g.
\]
\end{theorem}

\begin{proof}
Let $C:=v_1v_2\dots v_gv_1$ be a shortest cycle of $G$ and $S\subseteq V(G)\setminus V(C)$ with $|S|=k$. Let $H=G[V(C)\cup S]$. Let $B$ be the principal submatrix of $L(G)$ corresponding to vertices of $H$. Then $B=L(H)+M$, where
$M=\mbox{diag}\{d_{S'}(v):v\in V(C)\cup S\}$ with $S'=V(G)\setminus (V(C)\cup S)$. Evidently, $\rho_1(M)=\max\{d_{S'}(v): v\in V(C)\cup S\}\le |S'|=n-g-k$.

By Lemmas \ref{interlacing} and \ref{cw}, \[
\mu_{n-g+1}(G)=\rho_{n-(g+k)+k+1}(L(G))\le \rho_{k+1}(B)\le \mu_{k+1}(H)+\rho_1(M')\le \mu_{k+1}(H)+n-g-k.
\]
As $m_G[n-g-k+4,n]\le n-g$ if and only if $\mu_{n-g+1}(G)<n-g-k+4$, it suffices to show $\mu_{k+1}(H)<4$.

Suppose that $g=4$. Then $k=3$ and $H$ is a graph on $7$ vertices. If $H$ is bipartite, then
$H$ is a spanning subgraph of $K_{2,5}$ or $K_{3,4}$, as $\mu_4(K_{2,5})=2$ and $\mu_4(K_{3,4})=3$,
we have by Lemma \ref{cauchy} that $\mu_4(H)\le \max\{\mu_4(K_{2,5}), \mu_4(K_{3,4})\}<4$. Suppose that $H$ is not bipartite. Then $H$ contains cycles of length $5$. Then $H$ is a spanning subgraph of a graph obtained from $K_{2,4}$ by subdividing an edge, or from $K_{3,3}$ by subdividing an edge. So by Lemma \ref{cauchy}, $\mu_4(H)\le 2<4$ in the former case, and  $\mu_4(H)\le 3<4$ in the latter case.

Suppose that $g\ge 5$.
If there is a vertex, say $u\in S$, has two neighbors $v_1$ and $v_\ell$ on $C$, then $\ell+1\ge g$ and $g-\ell+3\ge g$, which implies that $g=4$, a contradiction.
So $d_C(v)\le 1$ for any $v\in S$.
Let
\begin{align*}
S_0&=\{v\in S: d_C(v)=1, d_S(v)=0\},\\
S_1&=\{v\in S: d_C(v)=1, d_S(v)\ge 1\},\\
S_2&=\{v\in S: d_C(v)=0, d_{S_1}(v)\ge 1\},\\
S_3&=S\setminus(S_0\cup S_1\cup S_2).
\end{align*}
Then $S_0$ is an independent set and $S_3$ contains precisely the vertices in $S$ that are not adjacent to any vertex on $C$ and in $S_1$.

\noindent
{\bf Case 1.} $S_1$ is independent.

Suppose that $S_1=\emptyset$. Then $S_2=\emptyset$. As $k\le g-1$, $G[S]=G[S_0\cup S_3]$ is a forest, so $|E(S)|\le |S_3|-1$. Let $E_0=E(S)\cup E(C,S)$. Then $|E_0|\le |S_3|-1+|S_0|\le k-1$ and $H-E_0\cong C_g\cup kK_1$.  By Lemmas \ref{cauchy} and \ref{sigma} (ii), \[
\mu_{k+1}(H)\le \mu_{k+1-|E_0|}(H-E_0)\le \mu_2(H-E_0)=\mu_2(C_g)<4.
\]

Suppose in the following that $S_1\ne \emptyset$.

\noindent
{\bf Case 1.1.} For any $v\in S_2$, $d_{S_1}(v)\le 2$.

As $S_1$ is independent, $N_S(w)\subseteq S_2$ for any $w\in S_1$. Assume that  $S_1=\{w_1,\dots,w_t\}$. Let $z_i$ be a neighbor of $w_i$ in $S_2$ and $E_1'=\{w_iz_i:i=1,\dots,t\}$.
Let $E_1=E(C,S)\cup E(S)\setminus E_1'$.
As $k\le g-1$, $|E(S)|\le k-|S_0|-1$, so $|E_1|=|E(C,S)|+|E(S)|-t\le k-1$. Note that $H-E_1\cong C_g\cup G[E_1']\cup (k-2t)K_1$ and the components $G[E_1']$ are $K_2$ or $K_{1,2}$. By Lemmas \ref{cauchy} and \ref{sigma} (ii) again, \[
\mu_{k+1}(H)\le \mu_2(H-E_1)\le \max\{\mu_2(C_g),\mu_1(K_2),\mu_1(K_{1,2})\}<4.
\]

\noindent
{\bf Case 1.2.} For some  $w\in S_2$, $d_{S_1}(w)\ge 3$.

Note that $k\ge 4$.
Let $u_1$, $u_2$ and $u_3$ be neighbors of $w$ in $S_1$. Assume that $u_1v_1\in E(G)$, $u_2v_i\in E(G)$ and $u_3v_j\in E(G)$  with $1\le i\le j\le g$.
As $g\ge 5$,  $1< i< j\le g$.
Assume that $i-1\le \min\{j-i,g+1-j\}$. Then $i\le \frac{g}{3}+1$, so $g\le i+3\le \frac{g}{3}+1+3$, implying that $g\le 6$.

Suppose that $g=5$. Then $i=2$, $j=3,4$, $k=4$ and $S=\{u_1,u_2,u_3,w\}$, so $H\cong Q_1,Q_2$ (see Fig. \ref{figq}). By a direct calculation, $\mu_5(Q_1)=2.555$ and $\mu_5(Q_2)=2$. So $\mu_5(H)<4$.

Suppose that $g=6$. Then $i=3$,  $j=5$, and $k=4,5$. If $k=4$, i.e., $S=\{u_1,u_2,u_3,w\}$, then $H\cong Q_3$ (see Fig. \ref{figq}), so we have by a direct calculation that  $\mu_5(H)=2<4$.  Otherwise, $k=5$, so there is exactly one vertex $z$ in $S$ different from $u_1,u_2,u_3$ and $w$ so that  $H-z\cong Q_3$. By Lemma \ref{interlacing}, $\mu_6(H)\le \mu_5(Q_3)+1=3<4$.

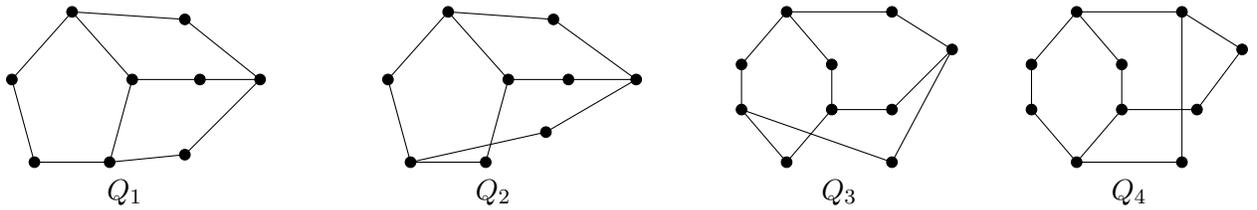
\begin{figure}[htbp]
\begin{minipage}[left]{.73\textwidth}
\begin{tikzpicture}
\filldraw [black] (-5,1) circle (2pt);
\filldraw [black] (-4.2,0.1) circle (2pt);
\filldraw [black] (-4.5,-1) circle (2pt);
\filldraw [black] (-5.8,0.1) circle (2pt);
\filldraw [black] (-5.5,-1) circle (2pt);
\filldraw [black] (-3.3,0.1) circle (2pt);
\filldraw [black] (-3.5,0.9) circle (2pt);
\filldraw [black] (-2.5,0.1) circle (2pt);
\filldraw [black] (-3.5,-0.9) circle (2pt);
\draw[black] (-2.5,0.1)--(-3.5,0.9)--(-5,1)--(-4.2,0.1)--(-4.5,-1)--(-5.5,-1)--(-5.8,0.1)--(-5,1);
\draw[black] (-2.5,0.1)--(-3.3,0.1)--(-4.2,0.1);
\draw[black] (-2.5,0.1)--(-3.5,-0.9)--(-4.5,-1);
\node at(-4.3,-1.4) {\small $Q_1$};
\filldraw [black] (0,1) circle (2pt);
\filldraw [black] (0.8,0.1) circle (2pt);
\filldraw [black] (0.5,-1) circle (2pt);
\filldraw [black] (-0.8,0.1) circle (2pt);
\filldraw [black] (-0.5,-1) circle (2pt);
\filldraw [black] (1.6,0.1) circle (2pt);
\filldraw [black] (1.4,0.9) circle (2pt);
\filldraw [black] (2.5,0.1) circle (2pt);
\filldraw [black] (1.3,-0.6) circle (2pt);
\draw[black] (0,1)--(0.8,0.1)--(0.5,-1)--(-0.5,-1)--(-0.8,0.1)--(0,1)--(1.4,0.9)--(2.5,0.1)--(1.6,0.1)--(0.8,0.1);
\draw[black] (-0.5,-1)--(1.3,-0.6)--(2.5,0.1);
\node at(0.6,-1.4) {\small $Q_2$};
\filldraw [black] (4.5,1) circle (2pt);
\filldraw [black] (4.5,-1) circle (2pt);
\filldraw [black] (5.1,0.3) circle (2pt);
\filldraw [black] (5.1,-0.3) circle (2pt);
\filldraw [black] (3.9,0.3) circle (2pt);
\filldraw [black] (3.9,-0.3) circle (2pt);
\filldraw [black] (5.9,1) circle (2pt);
\filldraw [black] (5.9,-0.3) circle (2pt);
\filldraw [black] (6.7,0.5) circle (2pt);
\filldraw [black] (5.9,-1) circle (2pt);
\draw[black] (4.5,1)--(5.1,0.3)--(5.1,-0.3)--(4.5,-1)--(3.9,-0.3)--(3.9,0.3)--(4.5,1);
\draw[black] (4.5,1)--(5.9,1)--(6.7,0.5)--(5.9,-0.3)--(5.1,-0.3);
\draw[black] (3.9,-0.3)--(5.9,-1)--(6.7,0.5);
\node at(5.2,-1.4) {\small $Q_3$};
\end{tikzpicture}
\end{minipage}
\begin{minipage}{.262\textwidth}
\hfill
\begin{tikzpicture}
\filldraw [black] (0,1) circle (2pt);
\filldraw [black] (0,-1) circle (2pt);
\filldraw [black] (0.6,0.3) circle (2pt);
\filldraw [black] (0.6,-0.3) circle (2pt);
\filldraw [black] (-0.6,0.3) circle (2pt);
\filldraw [black] (-0.6,-0.3) circle (2pt);
\filldraw [black] (1.4,1) circle (2pt);
\filldraw [black] (1.6,-0.3) circle (2pt);
\filldraw [black] (2.2,0.5) circle (2pt);
\filldraw [black] (1.4,-1) circle (2pt);
\draw[black] (0,1)--(0.6,0.3)--(0.6,-0.3)--(0,-1)--(-0.6,-0.3)--(-0.6,0.3)--(0,1);
\draw[black] (0,1)--(1.4,1)--(2.2,0.5)--(1.6,-0.3)--(0.6,-0.3);
\draw[black] (0,-1)--(1.4,-1)--(1.4,1);
\node at(0.7,-1.4) {\small $Q_4$};
\end{tikzpicture}
\end{minipage}
\caption{The graph $Q_1$, $Q_2$, $Q_3$ and $Q_4$.}
\label{figq}
\end{figure}

\noindent
{\bf Case 2.} $S_1$ is not independent.

In this case, assume that $u_1u_2\in E(G)$ with $u_1,u_2\in S$. Assume that $u_1v_1\in E(G)$ and $u_2v_i\in E(G)$. Assume that $i-1\le g+1-i$, i.e., $2i-2\le g$.  Then $i+2\ge g$ and $g-i+4\ge g$, so $i\le 4$ and $g\le 6$. As $g\ge 5$, we have $(i,g)=(4,6), (3,5)$.

\noindent
{\bf Case 2.1.} $(i, g)=(4,6)$.

Let $S_1'=\{v\in S_1:d_{S_1}(v)\ge 1\}$. Then $d_{S_1}(v)=1$ for each $v\in S_1'$. Otherwise, assume that $u_1$ has another neighbor $u_3\in S_1$.  As $g=6$, any vertex on $C$ can not be a neighbor of $u_3$, a contradiction.

Suppose that $S_2=\emptyset$. Then $S_1=S_1'$. Let $E_1=E(C,S)\cup E(S_3)$. As $g\le k-1$, $|E(S_3)|\le |S_3|-1$ and $|E_1|=|S_0|+|S_1|+|S_3|-1\le k-1$. Note that $H-E_1\cong C_6\cup \frac{|S_1|}{2}K_2\cup (k-|S_1|)K_1$. By Lemmas \ref{cauchy} and \ref{sigma} (ii), \[
\mu_{k+1}(H)\le \mu_2(H-E_1)=\mu_2(C_6)<4.
\]

Suppose that $S_2\ne \emptyset$. Any  $z\in S_2$  is adjacent to at most one of $u_1,u_2$.

Suppose that  for some $z\in S_2$, $d_{S_1}(z)\ge 2$.
If $N_{S_1}(z)\cap S_1'=\emptyset$, then, as $k\le g-1=5$, we have $k=5$, $d_{S_1}(z)=2$ and $S=\{u_1,u_2,z\}\cup N_{S_1}(z)$, so $|E(C,S)|=4$ and $H-E(C,S)\cong C_6\cup K_2\cup K_{1,2}$. So by Lemmas \ref{cauchy} and \ref{sigma} (ii), \[
\mu_{k+1}(H)=\mu_6(H)\le \mu_2(H-E(C,S))=\mu_2(C_6)<4.
\]
Suppose that  $N_{S_1}(z)\cap S_1'\ne \emptyset$.  Assume that $u_1z\in E(G)$. Let $u_3$ be the other neighbor of $z$ and let $v_j$ be the  neighbor of $u_3$ on $C$. Then $j=3,5$,
and $H[V(C)\cup \{u_1,u_2,u_3,z\}]\cong Q_4$
(see Fig. \ref{figq}). If $S=\{u_1,u_2,u_3,z\}$, then $\mu_{k+1}(H)=\mu_5(Q_4)=2.746<4$. Otherwise, there is exactly one vertex $z'$ in $S$ different from $u_1,u_2,u_3$ and $z$ such that $H-z'\cong Q_4$. It then follows by Lemma \ref{cw} that $\mu_6(H)\le \mu_5(Q_4)+1=3.746<4$.

Suppose that $d_{S_1}(w)=1$ for any $w\in S_2$. Let $E_2=(E(C,S)\setminus \{u_1v_1\})\cup E(S_1,S_2)$. Note that $|E_2|=|S_0|+|S_1|-1+|S_2|\le k-1$ and $H-E_2\cong U_8'\cup H'$, where $H'$ is a spanning subgraph of $G[S\setminus \{u_1,u_2\}]$. As $k\le 5$, $|V(H')|\le 3$.
By Lemmas \ref{cauchy} and \ref{u1},
\[
\mu_{k+1}(H)\le \mu_2(H-E_2)=\max\{\mu_2(U_8'),\mu_1(H')\}<4.
\]

\noindent
{\bf Case 2.2.} $(i, g)=(3,5)$.

Note that $k=3,4$, $|E(C,S)|\le k$ and $H-E(C,S)\cong C_5\cup G[S]$. If $d_S(v)\le 2$ for any $v\in S$, then, as $G[S]$ is a forest, $\mu_1(G[S])<4$, so we have by Lemmas \ref{cauchy} and \ref{sigma} (ii) that
\[
\mu_{k+1}(H)\le \mu_1(H-E(C,S))=\max\{\mu_1(C_5),\mu_1(G[S])\}<4.
\]
If there is a vertex $u\in S$ such that $d_S(u)=3$, then $k=4$ and $u\in \{u_1,u_2\}$. Assume that $u=u_1$. Denote by $u_3$ and $u_4$ the remaining two vertices in $S$. If both $u_3$ and $u_4$ have neighbors on $C$, then there would be a cycle with length less than $5$, a contradiction.
%
Thus,  one of $u_3$ and $u_4$, say $u_4$, has no neighbors on $C$. Let $E_3=E(C,S)\cup \{u_1u_4\}$. As $g\ge 5$, $S\setminus \{u_1\}$ is independent. Then $|E_3|\le 4$ and $H-E_3\cong C_5\cup K_{1,2}\cup K_1$. By Lemmas \ref{cauchy} and \ref{sigma} (ii),  \[
\mu_{k+1}(H)\le \mu_1(H-E_3)=\mu_1(C_5)<4.
\]

The proof is completed by combining Cases 1 and 2.
\end{proof}

\section{Proof of Theorem \ref{Thr}}

Now we turn to graphs with girth $3$ to prove Theorem \ref{Thr}.

\begin{proof}[Proof of Theorem \ref{Thr}]
From Lemma \ref{comp} and the fact that the
number of connected components of $\overline{G}$ is equal to $m_{\overline{G}}(0)$, we have
\begin{align*}
m_G(n)=n-3 &\Leftrightarrow\sigma_L(G)=\{n^{[n-3]}, a,b,0\} \mbox{ with some $a, b$ with $n>a\ge b>0$}\\
&\Leftrightarrow \sigma_L(\overline{G})=\{n-b, n-a, 0^{[n-2]}\}  \mbox{ with some $a, b$ with $n>a\ge b>0$}\\
&\Leftrightarrow
\overline{G}\cong (n-3)K_1\cup K_3, (n-3)K_1\cup P_3, (n-4)K_1\cup 2K_2\\
&\Leftrightarrow G\cong K_{n-3}\vee 3K_1, K_{n-3}\vee (K_1\cup K_2), K_{n-4}\vee 2K_2.
\end{align*}
The Item (i)  follows by noting  that $m_{K_n}(n)=n-1$ and $m_{K_n-e}(n)=n-2$.

In the following, we prove Item (ii).

It is trivial that $m_G[n-1,n]=n-1$ if and only if $G\cong K_n$. Suppose that $G\ncong K_n$. Let $d$ be the diameter of $G$. Then $2\le d\le n-2$.

\noindent
{\bf Case 1.} $d\ge 3$.

By Theorem 1.1 of \cite{XZ1},
if $d=3$, then $m_G[n-1,n]\le n-3$ with equality if and only if $G\cong  H(n), H(n,a)$ for some $a=1\dots,\lfloor\frac{n-4}{2}\rfloor$, and
if $d\ge 4$, then $m_G[n-1,n]\le m_G[n-d+2,n]\le n-4<n-3$.

\noindent
{\bf Case 2.} $d=2$.

By Theorem 5 of \cite{XZ},  $m_G[n-1,n]=n-2$ if and only if $G\cong K_n-\{vv_i:i=1,\dots,s\}$ for some $s=1,\dots,n-2$, where $v,v_1,\dots,v_s\in V(K_n)$.

Suppose that $G\not\cong K_n-\{vv_i:i=1,\dots,s\}$  for any $s=1,\dots,n-2$, where $v,v_1,\dots,v_s\in V(K_n)$. Then $m_G[n-1,n]\le n-3$.

\begin{Claim} \label{Fff}
If $G$ is not (isomorphic to) any of the graphs in (a), (b) or (c), then $\mu_{n-3}(G)<n-1$.
\end{Claim}

\begin{proof}
Suppose that $G$ is not a graph in (a), (b) and (c). Then we can view $G$ as a spanning subgraph of $K_n-\{vv_i:i=1,\dots,s\}-e-f$ for some $s=1,\dots,n-2$, where $e$ and $f$ are two edges different from any $vv_i$.  By Lemma \ref{comp}, it suffices to show that  $\mu_3(\overline{G})>1$.

If neither $e$ nor $f$ is incident to some $v_i$, then, as $G$ is not a graph in (a), (b) and (c), $e$ and $f$ are not adjacent, so $\overline{G}\cong K_{1,s}\cup 2K_2\cup (n-s-5)K_1$, and hence $\mu_3(\overline{G})=2>1$.

Suppose that exactly one of $e$ and $f$ is incident to some $v_i$, say $e$ is incident to $v_1$. If $e=v_1v_j$ for some $j=2,\dots,s$, then $\overline{G}\cong (K_1\vee (K_2\cup (s-2)K_1))\cup K_2\cup (n-s-3)K_1$, so  $\mu_3(\overline{G})=2>1$. If $e$ is not incident to any of $v_j$ with $j=2,\dots,s$, then $\overline{G}\cong K_{1,s}^+\cup K_2\cup (n-4-s)K_1$ if $e$ and $f$ are not adjacent, where $K_{1,s}^+$ is obtained by subdivision an edge of $K_{1,s}$, and $\overline{G}\cong K_{1,s}^{++}\cup (n-s-3)K_1$ otherwise, where $K_{1,s}^{++}$ is obtained  by adding an edge between a pendant vertex from $K_{1,s}$ and a vertex of  $K_2$.
By Lemma \ref{cauchy}, $\mu_3(\overline{G})\ge \min\{\mu_2(P_4),\mu_1(K_2)\}=2>1$ in the former case, and $\mu_3(\overline{G})\ge \mu_3(P_5)>1$ in the latter case.

Suppose that both $e$ and $f$ are incident to some of $v_1,\dots, v_s$.
Suppose first that both $e$ and $f$ are incident to two vertices of $v_1,\dots, v_s$, say $e=v_1v_2$ and $f=v_pv_q$ for some $p,q=1,\dots,s$. Then $\overline{G}\cong (K_1\vee (2K_2\cup (s-4)K_1))\cup (n-s-1)K_1$ if $\{p,q\}\cap \{1,2\}=\emptyset$ and $\overline{G}\cong (K_1\vee (K_{1,2}\cup (s-3)K_1))\cup (n-s-1)K_1$ otherwise, from which, by Lemma \ref{cauchy}, $\mu_3(\overline{G})\ge \mu_3(K_1\vee 2K_2)=2>1$ in the former case, and $\mu_3(\overline{G})\ge \mu_3(K_1\vee K_{1,2})=2>1$ in the latter case. Suppose next that at least one of $e$ and $f$ is incident to exactly one of $v_1,\dots, v_s$. Particularlly, if both $e$ and $f$ are incident to exactly one of  $v_1,\dots, v_s$, then as $G$ is not a graph in (a), (b) and (c), $e$ and $f$ are not adjacent. Then $\overline{G}$ contains $P_5$ as a subgraph, so it follows from Lemma \ref{cauchy} that $\mu_3(\overline{G})\ge \mu_3(P_5)>1$.
\end{proof}

\begin{Claim} \label{FI}
If $G$ is one of the graphs in (a), (b) or (c), then $m_G[n-1,n]=n-3$.
\end{Claim}

\begin{proof}
If $G$ is a graph in (a), i.e.,  $G\cong K_n-\{vv_i:i=1,\dots,s\}-v_1v_2$ for some $s=2,\dots,n-2$, where $v,v_1,\dots,v_s\in V(K_n)$, then $\overline{G}\cong (K_1\vee (K_2\cup (s-2)K_1))\cup (n-s-1)K_1$ and $\sigma_L(\overline{G})=\{s+1,3,1^{[s-2]},0^{[n-s]}\}$, so we have by Lemma \ref{comp} that $\sigma_L(G)=\{n^{[n-s-1]},n-1^{[s-2]},n-3,n-s-1,0\}$, implying  $m_G[n-1,n]=n-3$.

If $G$ is a graph in (b), i.e.,  $G\cong K_n-\{vv_i:i=1,\dots,s\}-\{uu_j:i=1,\dots,t\}$ for some $s=1,\dots,n-3$ and $t=1,\dots,\min\{s,n-2-s\}$, where $v,v_1,\dots,v_s,u,u_1,\dots,u_t\in V(K_n)$, then $\overline{G}\cong K_{1,s}\cup K_{1,t}\cup (n-2-s-t)K_1$ and $\sigma_L(\overline{G})=\{s+1,t+1,1^{[s+t-2]},0^{[n-s-t]}\}$, so by Lemma \ref{comp}, we have $\sigma_L(G)=\{n^{[n-s-t-1]},n-1^{[s+t-2]},n-s-1,n-t-1,0\}$, implying  $m_G[n-1,n]=n-3$.

Suppose that $G$ is a graph in (c), i.e.,  $G\cong K_n-\{vv_i:i=1,\dots,s\}-\{v_1u_j:j=1,\dots,t\}$ for some $s=2,\dots,n-2$ and $t=1,\dots,\min\{s-1,n-1-s\}$, where $v,v_1,\dots,v_s,u_1,\dots,u_t\in V(K_n)$. Then $\overline{G}\cong K_{1,s}\diamond K_{1,t}\cup (n-s-t-1)K_1$, where $K_{1,s}\diamond K_{1,t}$ is obtained from $K_{1,s}\cup K_{1,t}$ by identifying the vertex of degree $t$ of $K_{1,t}$ and a pendant vertex of $K_{1,s}$. Trivially, $m_{\overline{G}}(0)=n-s-t$. By Lemma \ref{chong}, $m_{\overline{G}}(1)\ge s+t-3$. By Lemma \ref{min}, $\mu_{n-1}(\overline{G})<1$. Thus $\mu_3(\overline{G})=1$. By Lemma \ref{comp}, $\mu_{n-3}(G)=n-1$, implying that $m_G[n-1,n]\ge n-3$, so   $m_G[n-1,n]\ge n-3$.
\end{proof}

By Claims \ref{Fff} and \ref{FI}, $m_G[n-1,n]=n-3$ if and only if $G$ is one of the graphs in (a), (b) or (c).

By combining Cases 1 and 2, the proof is completed.
\end{proof}

\section{Concluding Remarks}

In this paper, we studied the connections between girth and Laplacian eigenvalue distribution. From Theorems \ref{Gen} and \ref{Thr}, it is
a surprise that the girth $g$ has different roles in the Laplacian eigenvalue distribution when $g=3$ and when $g\ge 4$.

As $\mu_3(K_{3,4})=4$, $K_{3,4}$ attains the bound in Theorem \ref{Gen22} for  $k=3$.
For the graph $G_1$ in Fig. \ref{figg12}, 
since $\mu_3(G_1)=4$, it attains the bound in Theorem \ref{Gen22} for $k=3$.
For the graph $G_2$ in Fig. \ref{figg12}, since  
$\mu_4(G_2)=4$, it attains the bound in Theorem \ref{Gen22} for $k=4$. These examples demonstrate that the bound provided in Theorem \ref{Gen22} is sharp for $k=3,4$. It would be interesting to determine the graphs for which the bound in Theorem \ref{Gen22} is attained.

\begin{figure}[htbp]
\centering
\begin{tikzpicture}
\filldraw [black] (-1,0) circle (2pt);
\filldraw [black] (0,1) circle (2pt);
\filldraw [black] (0,-1) circle (2pt);
\filldraw [black] (1,1) circle (2pt);
\filldraw [black] (1,-1) circle (2pt);
\filldraw [black] (2,1) circle (2pt);
\filldraw [black] (2,-1) circle (2pt);
\filldraw [black] (2,0) circle (2pt);
\draw[black] (-1,0)--(0,1)--(1,1)--(2,1)--(2,0)--(-1,0)--(0,-1)--(1,-1)--(2,-1)--(2,0);
\draw[black] (1,1)--(1,-1);
\node at (0.5,-1.4) {\small $G_1$};
\filldraw [black] (6,-1) circle (2pt);
\filldraw [black] (6,1) circle (2pt);
\filldraw [black] (6,-0.3) circle (2pt);
\filldraw [black] (6,0.3) circle (2pt);
\filldraw [black] (5.5,-0.3) circle (2pt);
\filldraw [black] (5,-0.3) circle (2pt);
\filldraw [black] (6.5,-0.3) circle (2pt);
\filldraw [black] (7,-0.3) circle (2pt);
\filldraw [black] (5,0.3) circle (2pt);
\filldraw [black] (7,0.3) circle (2pt);
\draw[black] (6,1)--(7,0.3)--(7,-0.3)--(6,-1)--(5,-0.3)-- (5,0.3)--(6,1)--(6,-0.3)--(5,-0.3);
\draw[black] (7,-0.3)--(6,-0.3);
\node at (6,-1.4) {\small $G_2$};
\end{tikzpicture}
\caption{The graphs $G_1$ and $G_2$.}
\label{figg12}
\end{figure}
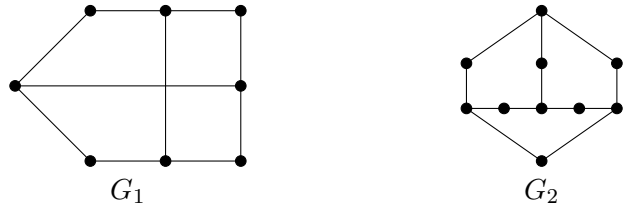

\bigskip

\noindent {\bf Declaration of competing interest}

There is no competing interest.

\end{document}